\documentclass[pre, amsmath, floatfix]{revtex4-2}

\usepackage{graphicx}
\usepackage{mathtools}
\usepackage{amsthm}

\theoremstyle{plain}
\newtheorem{theorem}{Theorem}[section]
\newtheorem{lemma}{Lemma}[section]
\newtheorem{proposition}{Proposition}[section]
\newtheorem{corollary}{Corollary}[section]

\theoremstyle{definition}
\newtheorem{definition}{Definition}[section]
\newtheorem{example}{Example}[section]

\theoremstyle{remark}
\newtheorem{remark}{Remark}[section]

\newcommand{\Prob}[1]{\mathrm{Pr}\left(#1\right)} 
\newcommand{\E}[1]{E\left[#1\right]} 
\newcommand{\Var}[1]{V\left[#1\right]} 
\DeclareMathOperator{\cov}{Cov} 
\newcommand{\Cov}[2]{\cov\left[#1, #2\right]}

\newcommand{\GaussFsymbol}{{}_2F_1} 
\newcommand{\GaussF}[4]{\GaussFsymbol\left(#1, #2; #3; #4\right)} 

\newcommand{\Omicron}{O}

\begin{document}
\title{Analysis of pairs in a generalized deck of playing cards}

\author{Ken Yamamoto}
\affiliation{Faculty of Science, University of the Ryukyus, Nishihara, Okinawa 903-0213, Japan}
\author{Satoshi Umeki}
\affiliation{Graduate School of Engineering and Science, University of the Ryukyus, Nishihara, Okinawa 903-0213, Japan}

\begin{abstract}
In this study, we analyze the number of pairs (i.e., two cards of the same rank) in a set of cards randomly selected from a deck of playing cards.
While the standard deck of playing cards comprises 52 cards excluding the joker with 4 suits and 13 ranks, our analysis considers a generalized deck where the numbers of suits and ranks can be set arbitrarily.
We derive the exact formulas for the mean and variance of the number of pairs in a given number of cards randomly selected from this generalized deck, expressed using the Gauss hypergeometric function.
Moreover, we derive the asymptotic behavior of the mean and variance as the number of ranks tends to infinity.
\end{abstract}

\maketitle

\section{Introduction}
Playing cards, along with dice and coins, are a standard object in problems of combinatorics and probability analysis~\cite{Epstein}.
They appear even in elementary exercises and learning material for mathematics~\cite{Knapp}.
The international standard deck of playing cards is a set of $52$ cards comprising $4$ suits (i.e., clubs, diamonds, hearts, and spades) and $13$ ranks (i.e., 2--9, jack, queen, king, and ace).
The coexistence of suits and ranks can yield complex situations.
For example, when five cards are dealt, the probability of obtaining the poker hand of four of a kind is~\cite{Epstein}
\[
\frac{\binom{13}{1}\binom{4}{4}\binom{12}{1}\binom{4}{1}}{\binom{52}{5}}=\frac{1}{4165}\approx2.4\times10^{-4}.
\]

Mathematical analysis related to playing cards has been particularly developed regarding shuffles.
A perfect riffle shuffle includes two different types: in- and out-shuffles.
The standard 52-card deck is known to return to its original order if perfect out-shuffles are repeated eight times~\cite{MorrisBook}.
The perfect riffle shuffle has been associated with dynamical systems~\cite{Scully} and group theory~\cite{Diaconis}.
Imperfect or random riffle shuffles have also been investigated.
The shuffling process is typically represented by the Gilbert--Shannon--Reeds model~\cite{Bayer}, where the rate of mixing is measured by the closeness to the uniform distribution of the permutations of the deck using the total variation distance.
Using this model, Bayer and Diaconis~\cite{Bayer} demonstrated that the 52-card deck becomes sufficiently mixed after approximately seven riffle shuffles.
Morris~\cite{Morris} investigated variations and generalization of the Gilbert--Shannon--Reeds model while other groups analyzed other shuffles such as overhand shuffle~\cite{Jonasson} and move-to-front shuffle~\cite{Fill}.
The mathematics of shuffling cards is comprehensively explained by Diaconis and Fulman~\cite{DiaconisBook}.

Mathematical analyses have been performed on card games such as poker~\cite{Bowling}, blackjack~\cite{Baldwin}, contract bridge~\cite{Frank}, baccarat~\cite{Kemeny}, and solitaire~\cite{Longpre}.
Even simple games such as memory have been shown to have certain nontrivial strategies for winning~\cite{Zwick}.
Mathematics has also been ingeniously applied to card tricks~\cite{Gardner}.

Rather than the standard $52$-card deck, our analysis considered a generalized deck of $N$ cards with $S(\ge2)$ suits, $R(\ge2)$ ranks, and no jokers, where $N=SR$.
The standard 52-card deck corresponds to $S=4$ and $R=13$, but other decks are used in some countries, e.g., a German deck of $40$ cards ($S=4$ and $R=10$) and a Spanish deck of $48$ cards ($S=4$ and $R=12$)~\cite{Britannica}.
Historically, decks of $S\ne4$ have sometimes been used.
For example, five-suit bridge was played by using the deck of $S=5$ and $R=13$.

In this study, we present a mathematical analysis on the \emph{pair}, which is defined as two cards with the same rank.
The pair plays a special role in card games such as poker, cribbage, old maid, and memory.
Let the random variable $X_n$ denote the number of pairs in $n$ cards randomly selected from a generalized $N$-card deck.
For example, in a standard 52-card deck, the set of five cards $\{3\clubsuit, 3\diamondsuit, 3\heartsuit, 4\heartsuit, 5\spadesuit\}$ has one pair:
\[
X_5(\{3\clubsuit, 3\diamondsuit, 3\heartsuit, 4\heartsuit, 5\spadesuit\})=1.
\]
The set of seven cards $\{\mathrm{J}\clubsuit, \mathrm{J}\diamondsuit, \mathrm{J}\heartsuit, \mathrm{J}\spadesuit, \mathrm{Q}\clubsuit, \mathrm{Q}\diamondsuit, \mathrm{K}\clubsuit\}$ has three pairs:
\[
X_7(\{\mathrm{J}\clubsuit, \mathrm{J}\diamondsuit, \mathrm{J}\heartsuit, \mathrm{J}\spadesuit, \mathrm{Q}\clubsuit, \mathrm{Q}\diamondsuit, \mathrm{K}\clubsuit\})=3.
\]
In terms of the number of pairs, the poker hands of two pairs, full house, and four of a kind all count as $X_5=2$.
In other words, these three poker hands have different strengths or occurrence probabilities~\cite{Epstein} but the same number of pairs.

To the best of the author's knowledge, although the pair is a simple concept, no thorough analysis of the pair beyond elementary exercises has been carried out.
This paper serves as the beginning of the in-depth mathematical analysis of the pair in playing cards.
As fundamental properties of the random variable $X_n$, the objective of this study is to derive its mean $\E{X_n}$ and variance $\Var{X_n}$ exactly and in closed forms, presented in Theorems~\ref{thm1} and \ref{thm2}, respectively.
Although the problem is simple, $\E{X_n}$ and $\Var{X_n}$ have nontrivial expressions involving the Gauss hypergeometric function.
Effective calculations can be performed by applying the properties of the multivariate hypergeometric distribution.
Theorem~\ref{thm3} states that $\E{X_n}/R$ and $\Var{X_n}/R$ remain finite in the $R\to\infty$ limit and provides their asymptotic forms.

\section{Main Theorems}
\begin{definition}[Gauss hypergeometric function~\cite{Olver}]
For any real number $\alpha$, the \emph{Pochhammer symbol} is defined as
\begin{equation}
(\alpha)_0=1,\quad (\alpha)_k=\alpha(\alpha+1)\cdots(\alpha+k-1).
\label{eq:Pochhammer}
\end{equation}
Using this symbol, the \emph{Gauss hypergeometric function} is defined as
\[
\GaussF{\alpha}{\beta}{\gamma}{z}\coloneqq\sum_{k=0}^\infty \frac{(\alpha)_k (\beta)_k}{(\gamma)_k}\frac{z^k}{k!}.
\]
\end{definition}

The mean and variance of the number of pairs $X_n$ can be expressed exactly using $\GaussFsymbol$.

\begin{theorem}
The mean of the number of pairs $X_n$ becomes
\[
\E{X_n}=
\begin{dcases}
\frac{n}{2}-\frac{R}{4}+\frac{R}{4}\frac{\binom{N-S}{n}}{\binom{N}{n}}\GaussF{-n}{-S}{N-S-n+1}{-1} & 0\le n\le\frac{N}{2},\\
\frac{n}{2}-\frac{R}{4}+(-1)^S\frac{R}{4}\frac{\binom{N-S}{N-n}}{\binom{N}{N-n}}\GaussF{n-N}{-S}{n-S+1}{-1} & \frac{N}{2}<n\le N.
\end{dcases}
\]
\label{thm1}
\end{theorem}

\begin{theorem}
The variance of the number of pairs $X_n$ for $R\ge4$ becomes
\[
\Var{X_n}=
\begin{dcases}
\frac{R}{16}+\frac{R(R-1)}{16}\frac{\binom{N-2S}{n}}{\binom{N}{n}}\GaussF{-n}{-2S}{N-2S-n+1}{-1}\\
\quad-\frac{R^2}{16}\left[\frac{\binom{N-S}{n}}{\binom{N}{n}}\GaussF{-n}{-S}{N-S-n+1}{-1}\right]^2 & \hspace{-1cm} 0\le n\le\frac{N}{2},\\
\frac{R}{16}+\frac{R(R-1)}{16}\frac{\binom{N-2S}{N-n}}{\binom{N}{N-n}}\GaussF{n-N}{-2S}{n-2S+1}{-1}\\
\quad-\frac{R^2}{16}\left[\frac{\binom{N-S}{N-n}}{\binom{N}{N-n}}\GaussF{n-N}{-S}{n-S+1}{-1}\right]^2 & \hspace{-1cm} \frac{N}{2}\le n\le N.
\end{dcases}
\]
For $R\le 3$ \textup{(}i.e., $R=3$ or $R=2$\textup{)},
\[
\Var{X_n}=
\begin{dcases}
\frac{R}{16}+(-1)^n\frac{R(R-1)}{16}\frac{\binom{2S}{n}}{\binom{N}{n}}\GaussF{-n}{2S-N}{2S-n+1}{-1}\\
\quad-\frac{R^2}{16}\left[\frac{\binom{N-S}{n}}{\binom{N}{n}}\GaussF{-n}{-S}{N-S-n+1}{-1}\right]^2 & \hspace{-3cm} 0\le n\le\frac{N}{2},\\
\frac{R}{16}+(-1)^{N-n}\frac{R(R-1)}{16}\frac{\binom{2S}{N-n}}{\binom{N}{N-n}}\GaussF{n-N}{2S-N}{n+2S-N+1}{-1}\\
\quad-\frac{R^2}{16}\left[\frac{\binom{N-S}{N-n}}{\binom{N}{N-n}}\GaussF{n-N}{-S}{n-S+1}{-1}\right]^2 & \hspace{-3cm} \frac{N}{2}\le n\le N.
\end{dcases}
\]
In both cases, the symmetry $\Var{X_n}=\Var{X_{N-n}}$ is obtained.
\label{thm2}
\end{theorem}

The proofs of these theorems are presented in Section~\ref{sec4}.

\begin{corollary}\label{cor2.1}
For any $S$ and $R$, the mean $\E{X_n}$ and variance $\Var{X_n}$ are polynomials of $n$ whose degrees are
\[
\deg(\E{X_n})=S,\quad
\deg(\Var{X_n})=2S.
\]
\end{corollary}

\begin{proof}
By the definition of the hypergeometric function and $(-S)_k=(-S)(-S+1)\cdots(-1)\cdot0\cdots(k-S-1)=0$ for $k\ge S+1$, we obtain
\begin{align*}
&\frac{\binom{N-S}{n}}{\binom{N}{n}}\GaussF{-n}{-S}{N-S-n+1}{-1}\\
&=\frac{(N-n)(N-n-1)\cdots(N-n-S+1)}{N(N-1)\cdots(N-S+1)}\sum_{k=0}^S\frac{(-1)^k(-n)_k(-S)_k}{(N-S-n+1)_k k!}\\
&=\frac{(N-S)!}{N!}\sum_{k=0}^S (-1)^k\binom{S}{k}(N-n)(N-n-1)\cdots(N-n-S+k+1)\\
&\qquad\qquad n(n-1)\cdots(n-k+1),
\end{align*}
where we define $n(n-1)\cdots(n-k+1)=1$ for $k=0$ and $(N-n)(N-n-1)\cdots(N-n-S+k+1)=1$ for $k=S$ based on the definition of the Pochhammer symbol in Eq.~\eqref{eq:Pochhammer}.
This is a polynomial of $n$ whose degree is obviously at most $S$.
The coefficient of $n^S$ becomes
\[
\frac{(N-S)!}{N!}\sum_{k=0}^S (-1)^k\binom{S}{k}(-1)^{S-k}
=(-1)^S\frac{(N-S)!}{N!}\sum_{k=0}^S \binom{S}{k}
=(-1)^S\frac{(N-S)!}{N!}2^S\ne0.
\]
Thus, $\deg(\E{X_n})=S$.

The same expression is obtained for the case of $N/2\le n\le N$:
\begin{align*}
&(-1)^S\frac{\binom{N-S}{N-n}}{\binom{N}{N-n}}\GaussF{n-N}{-S}{n-S+1}{-1}\\
&=\frac{(N-S)!}{N!}\sum_{k=0}^S (-1)^{S+k}\binom{S}{k}(N-n)(N-n-1)\cdots(N-n-k+1)\\
&\qquad\qquad n(n-1)\cdots(n-S+k+1)\\
&=\frac{(N-S)!}{N!}\sum_{k'=0}^S (-1)^{k'}\binom{S}{k'}(N-n)(N-n-1)\cdots(N-n-S+k'+1)\\
&\qquad\qquad n(n-1)\cdots(n-k'+1),
\end{align*}
where $k'=S-k$ is used in the last equality.

A similar calculation can be applied to $\Var{X_n}$.
\begin{align*}
&\frac{\binom{N-2S}{n}}{\binom{N}{n}}\GaussF{-n}{-2S}{N-2S-n+1}{-1}\\
&=\frac{(N-2S)!}{N!}\sum_{k=0}^{2S}(-1)^k\binom{2S}{k}(N-n)\cdots(N-2S-n+k+1)n\cdots(n-k+1).
\end{align*}
The coefficient of $n^{2S}$ in $\Var{X_n}$ becomes
\begin{align*}
&\frac{R(R-1)}{16}\frac{(N-2S)!}{N!}\sum_{k=0}^{2S}\binom{2S}{k}-\frac{R^2}{16}\left[\frac{(N-S)!}{N!}2^S\right]^2\\
&=\frac{R(R-1)}{16}\frac{(N-2S)!}{N!}4^S-\frac{R^2}{16}\frac{((N-S)!)^2}{(N!)^2}4^S\\
&=\frac{R(R-1)}{16}\frac{(N-2S)!}{N!}4^S\left[1-\frac{N-S-1}{N-1}\frac{N-S-2}{N-2}\cdots\frac{N-2S+1}{N-S+1}\right]>0.
\end{align*}
Thus, $\deg(\Var{X_n})=2S$.
\end{proof}

\begin{example}[$S=2$]\label{ex2.1}
When $S=2$, $\E{X_n}$ is a quadratic function of $n$ owing to Corollary~\ref{cor2.1}.
We immediately obtain $\E{X_0}=\E{X_1}=0$ because pairs cannot be created when $n\le1$.
When $n=N=2R$, $\E{X_N}=\E{X_{2R}}=R$.
A quadratic function that satisfies these three conditions can then be uniquely determined:
\begin{equation}
\E{X_n}=\frac{n^2-n}{2(2R-1)}.
\label{eq:E_S2}
\end{equation}

From Corollary~\ref{cor2.1}, $\Var{X_n}$ is a fourth-order polynomial of $n$.
Owing to the symmetry $\Var{X_n}=\Var{X_{2R-n}}$, we can assume $\Var{X_n}=a(n-R)^4+b(n-R)^2+c$, where $a$, $b$, and $c$ are functions of $R$ and not of $n$.
Just like for the mean, we obtain $\Var{X_0}=\Var{X_1}=0$.
In addition, when $n=2$, the pair number in $n=2$ cards is either $0$ or $1$.
Hence, $\E{X_2^2}=\E{X_2}=\Prob{X_2=1}$.
The event $X_2=1$ is attained if and only if the two selected cards have the same rank.
Therefore,
\begin{equation}
\Prob{X_2=1}=\frac{\binom{S}{2}R}{\binom{N}{2}}=\frac{S-1}{N-1}.
\label{eq:P_2}
\end{equation}
Note that this relation is valid for general $S$ and is not limited only to the case of $S=2$.
For $S=2$,
\[
\Var{X_2}=\Prob{X_2=1}-\Prob{X_2=1}^2=\frac{2(R-1)}{(2R-1)^2}.
\]
From the three conditions for $\Var{X_n}$ at $n=0$, $1$, and $2$, the coefficients $a$, $b$, and $c$ are determined.
Then, 
\begin{align}
\Var{X_n}&=\frac{(n-R)^4-(2R^2-2R+1)(n-R)^2+R^2(R-1)^2}{2(2R-1)^2(2R-3)}\nonumber\\
&=\frac{n(n-1)(n-2R)(n-2R+1)}{2(2R-1)^2(2R-3)}.
\label{eq:var_S2}
\end{align}

Equations~\eqref{eq:E_S2} and \eqref{eq:var_S2} can be obtained by setting $S=2$ in Theorems~\ref{thm1} and \ref{thm2}, but this calculation method is easier.
\end{example}

%

\begin{figure}\centering
\raisebox{3.8cm}{\small{(a)}}
\includegraphics[scale=0.8]{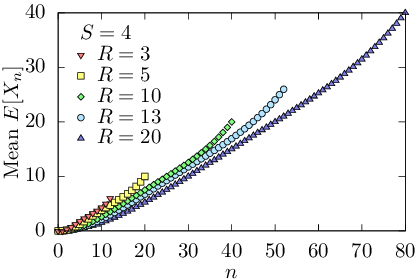}
\hspace{5mm}
\raisebox{3.8cm}{\small{(b)}}
\includegraphics[scale=0.8]{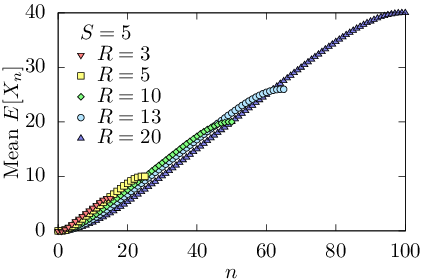}
\caption{
Numerical results for the mean $\E{X_n}$ of the number of pairs given in Theorem~\ref{thm1} as a function of $n$ for (a) $S=4$ and (b) $S=5$ at $R=3$, $5$, $10$, $13$, and $20$.
Circles in (a) (i.e., $S=4$ and $R=13$) correspond to the standard 52-card deck.
}
\label{fig1}
\end{figure}

The expressions of $\E{X_n}$ and $\Var{X_n}$ get more complicated for larger $S$, and their properties are difficult to understand for general $S$.
Figure~\ref{fig1} shows the numerical results for the mean $\E{X_n}$ as a function of $n$ with (a) $S=4$ and (b) $S=5$ at $R=3$, $5$, $10$, $13$, and $20$.
The numerical results in this study were obtained using Python, especially the hyp2f1 function from the SciPy library for the Gauss hypergeometric function.
From Theorem~\ref{thm1}, $\E{X_n}$ consists of three terms: the constant term $-R/4$, linear term $n/2$, and remaining hypergeometric term.
Owing to the $(-1)^S$ factor, the symmetry of the hypergeometric term changes depending on the parity of $S$.
It is symmetric with respect to $n$ and $N-n$ for even $S$, but it is antisymmetric for odd $S$.
Corresponding to this difference in symmetry, the shapes (i.e., convexity) of the right ends of the graphs are different for $S=4$ and $5$.

\begin{figure}\centering
\raisebox{3.8cm}{\small{(a)}}
\includegraphics[scale=0.8]{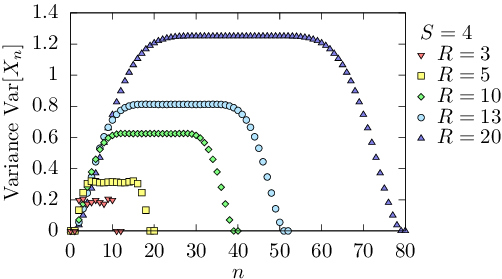}
\hspace{5mm}
\raisebox{3.8cm}{\small{(b)}}
\includegraphics[scale=0.8]{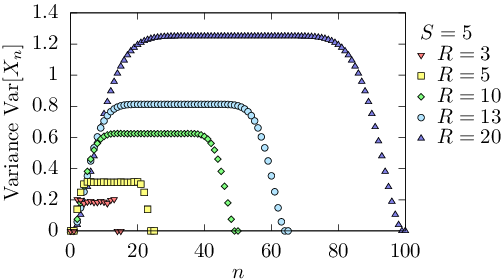}
\caption{
Numerical results for the variance $\Var{X_n}$ of the number of pairs given in Theorem~\ref{thm2} as a function of $n$ for (a) $S=4$ and (b) $S=5$ at $R=3$, $5$, $10$, $13$, and $20$.
Circles in (a) (i.e., $S=4$ and $R=13$) correspond to the standard 52-card deck.
}
\label{fig2}
\end{figure}

The numerical results for the variance $\Var{X_n}$ are shown in Fig.~\ref{fig2} for (a) $S=4$ and (b) $S=5$.
Each graph is symmetric with respect to the axis $n=N/2$.
Unlike the mean $\E{X_n}$ in Fig.~\ref{fig1}, the $\Var{X_n}$ graph has a similar shape regardless of whether  $S$ is even or odd.

To prove Theorem~\ref{thm1}, a simple approach is to find $\Prob{X_n=k}$ for each $k=0, 1,\ldots$ and to compute $\E{X_n}=\sum_k k\Prob{X_n=k}$.
In reality, however, $\Prob{X_n=k}$ is very difficult to calculate for arbitrary $n$, $k$, $S$, and $R$.
Although the case of $n=2$ can be calculated as in Eq.~\eqref{eq:P_2}, this is an exceptionally easy case.
Rather than calculating $\Prob{X_n=k}$, we prove Theorems~\ref{thm1} and \ref{thm2} by effectively utilizing the properties of the hypergeometric and multivariate hypergeometric distributions.

\section{Preliminary: hypergeometric and multivariate hypergeometric distributions}
Here, we review the hypergeometric distribution, which plays a crucial role in the proofs of our main theorems~\ref{thm1} and \ref{thm2}.
When $n$ cards are randomly selected from a deck of $N=SR$ cards, the number of rank-$i$ cards, represented by the random variable $Y_{n,i}$ in this study, follows the \emph{hypergeometric distribution}:
\begin{equation}
\Prob{Y_{n,i}=k}=\dfrac{\dbinom{S}{k}\dbinom{N-S}{n-k}}{\dbinom{N}{n}}
\label{eq:hg_pmf}
\end{equation}
for $i=1,\ldots,R$.
The random variables $Y_{n,1},\ldots, Y_{n,R}$ are not independent, and their joint distribution becomes
\[
\Prob{Y_{n,1}=k_1,\ldots,Y_{n,R}=k_R}=\frac{\dbinom{S}{k_1}\cdots\dbinom{S}{k_R}}{\dbinom{N}{n}},
\]
where the numerator represents the combination of ways to choose $k_i$ cards from the $S$ cards of rank $i$ for $i=1,\ldots,R$, subject to $k_1+\cdots+k_R=n$.
This distribution is an example of the \emph{multivariate hypergeometric distribution}~\cite{Bishop}, which is explained in detail by Balakrishnan~\cite{Balakrishnan}.
The basic properties of the multivariate hypergeometric distribution are presented in the following Proposition~\ref{prop3.1}.

\begin{proposition}[Basic properties of the multivariate hypergeometric distribution~\cite{Bishop, Yamamoto}]\label{prop3.1}
\renewcommand{\labelenumi}{$(\mathrm{\alph{enumi}})$}
\renewcommand{\theenumi}{\labelenumi}
The multivariate hypergeometric random variables $(Y_{n,1},\ldots,Y_{n,R})$ satisfy the following properties.
\begin{enumerate}
\item\label{prop3.1-a} The mean is expressed as
\begin{equation}
\E{Y_{n,i}}=\frac{nS}{N}=\frac{n}{R}.
\label{eq:hg_mean}
\end{equation}
\item The variance is expressed as
\[
\Var{Y_{n,i}}=\frac{n(R-1)(N-n)}{R^2(N-1)}.
\]
\item The covariance is expressed as
\[
\Cov{Y_{n,i}}{Y_{n,j}}=-\frac{n(N-n)}{R^2(N-1)}
\]
for $i\ne j$.
\item The probability-generating function is expressed using the Gauss hypergeometric function:
\begin{subequations}\label{eq:pgf}
\begin{align}
G(z)&\coloneqq\E{z^{Y_{n,i}}}=\sum_{k=0}^S z^k \Prob{Y_{n,i}=k}\nonumber\\
& && 0\le n\le \dfrac{N}{2},\label{eq:pgf-1}\\
&=
\smash{\left\{\begin{array}{@{}l@{}}
\dfrac{\binom{N-S}{n}}{\binom{N}{n}}\GaussF{-n}{-S}{N-S-n+1}{z}\\[5\jot]
\dfrac{\binom{N-S}{N-n}}{\binom{N}{N-n}}z^S\GaussF{n-N}{-S}{n-S+1}{z^{-1}}
\end{array}\right.} &&\notag\\[-\jot]
& && \dfrac{N}{2}\le n\le N.\label{eq:pgf-2}
\end{align}
\end{subequations}
\end{enumerate}

\end{proposition}

\begin{remark}
The first and second arguments of the hypergeometric functions in Eq.~\eqref{eq:pgf} (i.e., $-n$, $-S$, and $n-N$) are all negative.
Hence, the infinite sum of $\GaussFsymbol$ terminates at a finite term, and $G(z)$ is always a polynomial of $z$.
\end{remark}

\begin{remark}
In many books such as Bishop et al.~\cite{Bishop}, the generating functions of the hypergeometric distribution (e.g., the probability- and moment-generating functions and the characteristic function) appear only in a form corresponding to Eq.~\eqref{eq:pgf-1}, and Eq.~\eqref{eq:pgf-2} is not mentioned.
The hypergeometric function in Eq.~\eqref{eq:pgf-1} becomes divergent when the third argument $N-S-n+1$ of $\GaussFsymbol$ is not positive (i.e., when $n\ge N-S+1$) because $(N-S-n+1)_k=0$ for $k\ge S+n-N$.
Meanwhile, the binomial coefficient $\binom{N-S}{n}$ vanishes when $n\ge N-S+1$, and Eq.~\eqref{eq:pgf-1} takes the indeterminate form of $0\cdot\infty$.
The determinate and correct formula of $G(z)$ for $n\ge N-S+1$ is given by Eq.~\eqref{eq:pgf-2}.
For $S\le n\le N-S$, both Eqs.~\eqref{eq:pgf-1} and \eqref{eq:pgf-2} can be defined and become identical.
The boundary $N/2$ between Eqs.~\eqref{eq:pgf-1} and \eqref{eq:pgf-2} is selected simply due to symmetry in appearance and not essential.
See Yamamoto~\cite{Yamamoto} for the correct formulas for generating functions of the hypergeometric distribution in detail.
\end{remark}

To prove Theorems~\ref{thm1} and \ref{thm2}, we introduce two more properties of the multivariate hypergeometric distribution in addition to the basic properties in Proposition~\ref{prop3.1}.

\begin{proposition}\label{prop3.4}
\begin{enumerate}
\renewcommand{\labelenumi}{$(\mathrm{\alph{enumi}})$}
\renewcommand{\theenumi}{\labelenumi}
\item\label{prop3.4-a}
The conditional probability of $Y_{n,i}=k$ under the condition $Y_{n,j}=\ell$ \textup{(}$i\ne j$\textup{)} becomes
\[
\Prob{Y_{n,i}=k\mid Y_{n,j}=\ell}=\frac{\dbinom{S}{k}\dbinom{N-2S}{n-k-\ell}}{\dbinom{N-S}{n-\ell}}.
\]
That is, the conditional probability distribution of $Y_{n,i}$ given $Y_{n,j}$ is a hypergeometric distribution where the parameters $N$ and $n$ in Eq.~\eqref{eq:hg_pmf} are replaced with $N-S$ and $n-Y_{n,j}$, respectively.

The corresponding conditional mean becomes
\[
\E{Y_{n,i}\mid Y_{n,j}=\ell}=\frac{n-\ell}{R-1},
\]
or
\[
\E{Y_{n,i}\mid Y_{n,j}}=\frac{n-Y_{n,j}}{R-1}.
\]

\item\label{prop3.4-b}
The sum of two variables in the multivariate hypergeometric distribution follows a hypergeometric distribution:
\[
\Prob{Y_{n,i}+Y_{n,j}=k}=\frac{\dbinom{2S}{k}\dbinom{N-2S}{n-k}}{\dbinom{N}{n}}.
\]
\end{enumerate}
The parameter $S$ in Eq.~\eqref{eq:hg_pmf} is replaced with $2S$.
\end{proposition}

\begin{proof}
\ref{prop3.4-a}
When there are $\ell$ cards of rank $j$ in $n$ cards, the remaining $n-\ell$ cards are randomly selected from $N-S$ cards, excluding rank-$j$ cards.
Hence, the conditional probability $\Prob{Y_{n,i}=k\mid Y_{n,j}=\ell}$ can be obtained by replacing $N$ and $n$ in Eq.~\eqref{eq:hg_pmf} with $N-S$ and $n-\ell$, respectively.

We can provide another proof, based on the definition of the conditional probability
\[
\Prob{Y_{n,i}=k\mid Y_{n,j}=\ell}=\frac{\Prob{Y_{n,i}=k, Y_{n,j}=\ell}}{\Prob{Y_{n,j}=\ell}},
\]
with
\[
\Prob{Y_{n,j}=\ell}=\frac{\dbinom{S}{\ell}\dbinom{N-S}{n-\ell}}{\dbinom{N}{n}},\quad
\Prob{Y_{n,i}=k, Y_{n,j}=\ell}=\frac{\dbinom{S}{k}\dbinom{S}{\ell}\dbinom{N-2S}{n-k-\ell}}{\dbinom{N}{n}}.
\]

The conditional mean $\E{Y_{n,i}\mid Y_{n,j}=\ell}$ is immediately obtained by replacing $R$ and $n$ in Eq.~\eqref{eq:hg_mean} with $R-1$ and $n-\ell$, respectively.

\ref{prop3.4-b}
$Y_{n,i}+Y_{n,j}$ represents the number of rank-$i$ or rank-$j$ cards in $n$ selected cards.
As the deck contains $2S$ cards of rank $i$ or $j$, \ref{prop3.4-b} is proven.

Alternatively,
\begin{align*}
\Prob{Y_{n,i}+Y_{n,j}=k}&=\sum_{\ell=0}^k \Prob{Y_{n,i}=\ell, Y_{n,j}=k-\ell}\\
&=\sum_{\ell=0}^k \frac{\dbinom{S}{\ell}\dbinom{S}{k-\ell}\dbinom{N-2S}{n-k}}{\dbinom{N}{n}}
=\frac{\dbinom{2S}{k}\dbinom{N-2S}{n-k}}{\dbinom{N}{n}},
\end{align*}
where
\[
\sum_{\ell=0}^k \binom{S}{\ell}\binom{S}{k-\ell}=\binom{2S}{k}
\]
is a special case of Vandermonde's convolution~\cite{Riordan}.
\end{proof}

According to Yamamoto~\cite{Yamamoto}, the probability-generating function of $Y_{n,i}+Y_{n,j}$, represented by $G_2(z)$ in this study, can be written explicitly.
For $R\ge 4$,
\begin{align}
G_2(z)=
\begin{dcases}
\dfrac{\binom{N-2S}{n}}{\binom{N}{n}}\GaussF{-n}{-2S}{N-2S-n+1}{z} & 0\le n\le\frac{N}{2},\\
\dfrac{\binom{N-2S}{N-n}}{\binom{N}{N-n}}z^{2S}\GaussF{n-N}{-2S}{n-2S+1}{z^{-1}} & \frac{N}{2}\le n\le N.
\end{dcases}
\label{eq:G2-1}
\end{align}
For $R\le 3$ (i.e., $R=3$ or $R=2$),
\begin{align}
G_2(z)=
\begin{dcases}
\dfrac{\binom{2S}{n}}{\binom{N}{n}}z^n\GaussF{-n}{2S-N}{2S-n+1}{z^{-1}} & 0\le n\le\frac{N}{2},\\
\dfrac{\binom{2S}{N-n}}{\binom{N}{N-n}}z^{n+2S-N}\GaussF{n-N}{2S-N}{n+2S-N+1}{z} & \frac{N}{2}\le n\le N.
\end{dcases}
\label{eq:G2-2}
\end{align}
In particular, when $R=2$, $G_2(z)$ is simplified to $G_2(z)=z^n$ for any $0\le n\le N$.
The expression of $G_2(z)$ changes depending on $R\ge4$ or $R\le3$; this is owing to whether the maximum value $2S$ of $Y_{n,i}+Y_{n,j}$ exceeds half of the total number $N$ of cards~\cite{Yamamoto}.

\section{Proof of main theorems}\label{sec4}
Here, we present the proofs of Theorems~\ref{thm1} and \ref{thm2}.

By the definition of pair, the following important relation between the number of pairs $X_n$ and the multivariate hypergeometric variable $Y_{n,i}$ ($i=1,\ldots,R$) can be obtained.

\begin{proposition}\label{prop4.1}
\[
X_n=\sum_{i=1}^R \left\lfloor\frac{Y_{n,i}}{2}\right\rfloor.
\]
\end{proposition}

According to this proposition, the properties of $X_n$ can be reduced to those of the random variable $\lfloor Y_{n,i}/2\rfloor$.

\begin{lemma}[Properties of $\lfloor Y_{n,i}/2\rfloor$]\label{lemma1}
The mean and variance of $\lfloor Y_{n,i}/2\rfloor$ for $i=1,\ldots,R$) and the covariance of $\lfloor Y_{n,i}/2\rfloor$ and $\lfloor Y_{n,j}/2\rfloor$ \textup{(}$i\ne j$\textup{)} can be written by using the probability-generating functions $G(z)$ in Eq.~\eqref{eq:pgf} and $G_2(z)$ in Eqs.~\eqref{eq:G2-1} and \eqref{eq:G2-2}.
\begin{enumerate}
\renewcommand{\labelenumi}{$\mathrm{(\alph{enumi})}$}
\renewcommand{\theenumi}{\labelenumi}
\item\label{itm:lemma1-a} The mean of $\lfloor Y_{n,i}/2\rfloor$ can be expressed by $G(z)$ at $z=-1$:
\[
\E{\left\lfloor\dfrac{Y_{n,i}}{2}\right\rfloor}=\frac{n}{2R}-\frac{1}{4}+\frac{G(-1)}{4}.
\]
\item\label{itm:lemma1-b} The variance of $\lfloor Y_{n,i}/2\rfloor$ can be expressed by $G(z)$ and its derivative $G'(z)$ at $z=-1$:
\[
\Var{\left\lfloor\dfrac{Y_{n,i}}{2}\right\rfloor}=\frac{n(R-1)(N-n)}{4R^2(N-1)}+\frac{1}{16}-\frac{n}{4R}G(-1)-\frac{(G(-1))^2}{16}-\frac{G'(-1)}{4}.
\]
\item\label{itm:lemma1-c} The covariance of $\lfloor Y_{n,i}/2\rfloor$ and $\lfloor Y_{n,j}/2\rfloor$ \textup{(}$i\ne j$\textup{)} can be expressed by combining $G_2(-1)$ with $G(-1)$ and $G'(-1)$:
\[
\Cov{\left\lfloor\dfrac{Y_{n,i}}{2}\right\rfloor}{\left\lfloor\dfrac{Y_{n,j}}{2}\right\rfloor}=-\frac{n(N-n)}{4R^2(N-1)}+\frac{nG(-1)}{4R(R-1)}-\frac{(G(-1))^2}{16}+\frac{G'(-1)}{4(R-1)}+\frac{G_2(-1)}{16}.
\]
\end{enumerate}
\end{lemma}

\begin{remark}
From Proposition~\ref{prop3.1}, we obtain
\begin{align*}
&\E{\frac{Y_{n,i}}{2}}=\frac{n}{2R},\quad
\Var{\frac{Y_{n,i}}{2}}=\frac{n(R-1)(N-n)}{4R^2(N-1)},\\
&\Cov{\frac{Y_{n,i}}{2}}{\frac{Y_{n,j}}{2}}=-\frac{n(N-n)}{4R^2(N-1)}.
\end{align*}
These results appear as the first terms in Lemma~\ref{lemma1}\ref{itm:lemma1-a}--\ref{itm:lemma1-c}.
The remaining complicated terms in Lemma~\ref{lemma1}, which include hypergeometric functions in $G$ and $G_2$, can be attributed to the floor function $\lfloor\cdot\rfloor$.
\end{remark}

\begin{proof}
First, the condition $\lfloor Y_{n,i}/2\rfloor=k$ is satisfied only when $Y_{n,i}=2k$ or $Y_{n,i}=2k+1$.
Hence,
\[
\Prob{\left\lfloor\frac{Y_{n,i}}{2}\right\rfloor=k}=\Prob{Y_{n,i}=2k}+\Prob{Y_{n,i}=2k+1}.
\]

\ref{itm:lemma1-a}
The mean of $\lfloor Y_{n,i}/2\rfloor$ can be calculated as follows:
\begin{align*}
\E{\left\lfloor\dfrac{Y_{n,i}}{2}\right\rfloor}&=\sum_{k=0}^{\lfloor S/2\rfloor}k\Prob{\left\lfloor\frac{Y_{n,i}}{2}\right\rfloor=k}
=\sum_{k=0}^{\lfloor S/2\rfloor}k[\Prob{Y_{n,i}=2k}+\Prob{Y_{n,i}=2k+1}]\\
&=\frac{1}{2}\left(\E{Y_{n,i}}-\sum_{k=0}^{\lfloor S/2\rfloor}\Prob{Y_{n,i}=2k+1}\right).
\end{align*}
According to Proposition~\ref{prop3.1}\ref{prop3.1-a}, the first term can be immediately calculated as $\E{Y_{n,i}}=n/R$.
We then need to calculate the second term, which represents the probability of $Y_{n,i}$ being odd.

We obtain
\begin{equation}
\sum_{k=0}^{\lfloor S/2\rfloor}\Prob{Y_{n,i}=2k}+\sum_{k=0}^{\lfloor S/2\rfloor}\Prob{Y_{n,i}=2k+1}=G(1)=1,
\label{eq:Y_normalize}
\end{equation}
and
\begin{equation}
\sum_{k=0}^{\lfloor S/2\rfloor}\Prob{Y_{n,i}=2k}-\sum_{k=0}^{\lfloor S/2\rfloor}\Prob{Y_{n,i}=2k+1}=\E{(-1)^{Y_{n,i}}}=G(-1),
\label{eq:Y_difference}
\end{equation}
where $G$ is the probability-generating function of $Y_{n,i}$ given in Eq.~\eqref{eq:pgf}.
By subtracting Eq.~\eqref{eq:Y_difference} from Eq.~\eqref{eq:Y_normalize}, we obtain
\begin{equation}
\sum_{k=0}^{\lfloor S/2\rfloor}\Prob{Y_{n,i}=2k+1}=\frac{1-G(-1)}{2}.
\label{eq:odd}
\end{equation}
Therefore, \ref{itm:lemma1-a} is proven.

\ref{itm:lemma1-b}
The variance is defined by
\[
\Var{\left\lfloor\frac{Y_{n,i}}{2}\right\rfloor}=\E{\left\lfloor\frac{Y_{n,i}}{2}\right\rfloor^2}-\E{\left\lfloor\frac{Y_{n,i}}{2}\right\rfloor}^2.
\]
The second term is given by the square of \ref{itm:lemma1-a}.
The first term can be calculated as follows:
\begin{align}
\E{\left\lfloor\frac{Y_{n,i}}{2}\right\rfloor^2}&=\sum_{k=0}^{\lfloor S/2\rfloor}k^2[\Prob{Y_{n,i}=2k}+\Prob{Y_{n,i}=2k+1}]\nonumber\\
&=\frac{1}{4}\E{Y_{n,i}^2}-\frac{1}{2}\sum_{k=0}^{\lfloor S/2\rfloor}(2k+1)\Prob{Y_{n,i}=2k+1}+\frac{1}{4}\sum_{k=0}^{\lfloor S/2\rfloor}\Prob{Y_{n,i}=2k+1}
\label{eq:E_YY}
\end{align}
The second term of Eq.~\eqref{eq:E_YY} can be calculated in a similar manner as for Eq.~\eqref{eq:odd}.
The mean of $Y_{n,i}$ is then expressed as
\[
\sum_{k=0}^{\lfloor S/2\rfloor}2k\Prob{Y_{n,i}=2k}+\sum_{k=0}^{\lfloor S/2\rfloor}(2k+1)\Prob{Y_{n,i}=2k+1}=\E{Y_{n,i}}=\frac{n}{R}.
\]
If the sign of the second term is reversed, then
\begin{align*}
&\sum_{k=0}^{\lfloor S/2\rfloor}2k\Prob{Y_{n,i}=2k}-\sum_{k=0}^{\lfloor S/2\rfloor}(2k+1)\Prob{Y_{n,i}=2k+1}\\
&=\E{(-1)^{Y_{n,i}}Y_{n,i}}=\left.-\frac{d}{dz}\E{z^{Y_{n,i}}}\right|_{z=-1}=-G'(-1).
\end{align*}
From these two equations, we obtain
\begin{equation}
\sum_{k=0}^{\lfloor S/2\rfloor}(2k+1)\Prob{Y_{n,i}=2k+1}=\frac{n}{2R}+\frac{G'(-1)}{2}.
\label{eq:oddsum_kP}
\end{equation}
The third term in Eq.~\eqref{eq:E_YY} is calculated by using Eq.~\eqref{eq:odd}.
Therefore, \ref{itm:lemma1-b} is proven.

Although the specific form of $G'(-1)$ is not necessary for $\E{X_n}$ and $\Var{X_n}$, we calculate $G'(-1)$ here as a reference.
The differentiation formula of the hypergeometric function is given by~\cite{Olver}
\[
\frac{d}{dz}\GaussF{\alpha}{\beta}{\gamma}{z}=\frac{\alpha\beta}{\gamma}\GaussF{\alpha+1}{\beta+1}{\gamma+1}{z}.
\]
We can use this formula to obtain
\[
G'(-1)=
\begin{dcases}
\dfrac{\binom{N-S}{n-1}}{\binom{N}{n}}S\GaussF{-n+1}{-S+1}{N-S-n+2}{-1} & 0\le n\le \dfrac{N}{2},\\
\dfrac{\binom{N-S}{N-n-1}}{\binom{N}{N-n}}(-1)^S S\GaussF{n-N+1}{-S+1}{n-S+2}{-1} & \dfrac{N}{2}\le n\le N.
\end{dcases}
\]

\ref{itm:lemma1-c}
The covariance is defined as
\begin{equation}
\Cov{\left\lfloor\frac{Y_{n,i}}{2}\right\rfloor}{\left\lfloor\frac{Y_{n,j}}{2}\right\rfloor}=\E{\left\lfloor\frac{Y_{n,i}}{2}\right\rfloor\left\lfloor\frac{Y_{n,j}}{2}\right\rfloor}-\E{\left\lfloor\frac{Y_{n,i}}{2}\right\rfloor}\E{\left\lfloor\frac{Y_{n,j}}{2}\right\rfloor}.
\label{eq:cov_def}
\end{equation}
The second term on the right-hand side can be calculated by using \ref{itm:lemma1-a}.

We then need to compute
\begin{align*}
\E{\left\lfloor\frac{Y_{n,i}}{2}\right\rfloor \left\lfloor\frac{Y_{n,j}}{2}\right\rfloor}
&=\sum_{k=0}^{\lfloor S/2\rfloor}\sum_{\ell=0}^{\lfloor S/2\rfloor}k\ell[\Prob{Y_{n,i}=2k, Y_{n,j}=2\ell}+\Prob{Y_{n,i}=2k, Y_{n,j}=2\ell+1}\\
&\qquad\qquad+\Prob{Y_{n,i}=2k+1, Y_{n,j}=2\ell}+\Prob{Y_{n,i}=2k+1, Y_{n,j}=2\ell+1}].
\end{align*}
To simplify the notation, we introduce
\[
p_{k, \ell}=\Prob{Y_{n,i}=k, Y_{n,j}=\ell}.
\]
Consequently,
\begin{align}
&\E{\left\lfloor\frac{Y_{n,i}}{2}\right\rfloor \left\lfloor\frac{Y_{n,j}}{2}\right\rfloor}\nonumber\\
&=\frac{1}{4}\sum_{k=0}^{\lfloor S/2\rfloor}\sum_{\ell=0}^{\lfloor S/2\rfloor}(2k)(2\ell)p_{2k, 2\ell}
+\frac{1}{4}\sum_{k=0}^{\lfloor S/2\rfloor}\sum_{\ell=0}^{\lfloor S/2\rfloor}\left[(2k)(2\ell+1)p_{2k, 2\ell+1}-2kp_{2k, 2\ell+1}\right]\nonumber\\
&\quad+\frac{1}{4}\sum_{k=0}^{\lfloor S/2\rfloor}\sum_{\ell=0}^{\lfloor S/2\rfloor}\left[(2k+1)(2\ell)p_{2k+1, 2\ell}-2\ell p_{2k+1, 2\ell}\right]\nonumber\\
&\quad+\frac{1}{4}\sum_{k=0}^{\lfloor S/2\rfloor}\sum_{\ell=0}^{\lfloor S/2\rfloor}\left[(2k+1)(2\ell+1)p_{2k+1, 2\ell+1}-(2k+1)p_{2k+1, 2\ell+1}\right.\nonumber\\
&\qquad\left.\qquad -(2\ell+1)p_{2k+1, 2\ell+1}+p_{2k+1, 2\ell+1}\right]\nonumber\\
&=\frac{1}{4}\E{Y_{n,i}Y_{n,j}}-\frac{1}{4}\sum_{k=0}^S\sum_{\ell=0}^{\lfloor S/2\rfloor}kp_{k, 2\ell+1}-\frac{1}{4}\sum_{k=0}^{\lfloor S/2\rfloor}\sum_{\ell=0}^{S}\ell p_{2k+1, \ell}+\frac{1}{4}\sum_{k=0}^{\lfloor S/2\rfloor}\sum_{\ell=0}^{\lfloor S/2\rfloor}p_{2k+1, 2\ell+1}.
\label{eq:E_YiYj}
\end{align}
Owing to the symmetry of $p_{k, 2\ell+1}=p_{2k+1, \ell}$, the second and third terms in Eq.~\eqref{eq:E_YiYj} become identical.
By using the conditional probability
\[
p_{k, 2\ell+1}=\Prob{Y_{n,i}=k, Y_{n,j}=2\ell+1}=\Prob{Y_{n,j}=2\ell+1}\Prob{Y_{n,i}=k\mid Y_{n,j}=2\ell+1},
\]
the second term in Eq.~\eqref{eq:E_YiYj} can be calculated as
\begin{align*}
\sum_{k=0}^S\sum_{\ell=0}^{\lfloor S/2\rfloor}kp_{k, 2\ell+1}
&=\sum_{\ell=0}^{\lfloor S/2\rfloor}\Prob{Y_{n,j}=2\ell+1}\E{Y_{n,i}\mid Y_{n,j}=2\ell+1}\\
&=\sum_{\ell=0}^{\lfloor S/2\rfloor}\Prob{Y_{n,j}=2\ell+1}\frac{n-2\ell-1}{R-1}
=\frac{n}{2R}-\frac{nG(-1)+G'(-1)}{2(R-1)},
\end{align*}
where the conditional mean provided in Proposition~\ref{prop3.4}\ref{prop3.4-a} is applied in the second equality and Eqs.~\eqref{eq:odd} and \eqref{eq:oddsum_kP} are used in the last equality.
The remaining term in Eq.~\eqref{eq:E_YiYj} is
\[
\frac{1}{4}\sum_{k=0}^{\lfloor S/2\rfloor}\sum_{\ell=0}^{\lfloor S/2\rfloor}p_{2k+1, 2\ell+1}\eqqcolon \frac{1}{4}p_{\text{odd},\text{odd}},
\]
where $p_{\text{odd},\text{odd}}$ is the probability that both the numbers of rank-$i$ and rank-$j$ cards in $n$ cards are odd.
Similarly, we introduce $p_{\text{odd},\text{even}}$ and $p_{\text{even},\text{odd}}$.
These three probabilities satisfy the following relations:
\begin{align*}
p_{\text{odd},\text{odd}}+p_{\text{odd},\text{even}}&=\Prob{\text{$Y_{n,i}$ is odd}}=\frac{1-G(-1)}{2},\\
p_{\text{odd},\text{odd}}+p_{\text{even},\text{odd}}&=\Prob{\text{$Y_{n,j}$ is odd}}=\frac{1-G(-1)}{2}.
\end{align*}
According to Proposition~\ref{prop3.4}\ref{prop3.4-b}, $Y_{n,i}+Y_{n,j}$ follows the hypergeometric distribution.
Thus, the following relation can be obtained:
\[
p_{\text{odd},\text{even}}+p_{\text{even},\text{odd}}=\Prob{\text{$Y_{n,i}+Y_{n,j}$ is odd}}=\frac{1-G_2(-1)}{2}.
\]
Therefore,
\[
p_{\text{odd},\text{odd}}=\frac{1-G(-1)}{2}-\frac{1-G_2(-1)}{4}.
\]
All the sums in Eq.~\eqref{eq:E_YiYj} can be calculated, and \ref{itm:lemma1-c} can be proven by substituting these results into Eq.~\eqref{eq:cov_def}.
\end{proof}

Theorems~\ref{thm1} and \ref{thm2} are immediately proven by using Lemma~\ref{lemma1}.

\begin{proof}[Proof of Theorem~\ref{thm1}]
From Lemma~\ref{lemma1}\ref{itm:lemma1-a} and Eq.~\eqref{eq:pgf},
\begin{align}
\E{X_n}&=\sum_{i=1}^R \E{\left\lfloor\frac{Y_{n,i}}{2}\right\rfloor}=\frac{n}{2}-\frac{R}{4}+\frac{RG(-1)}{4}\label{eq:proof_E}\\
&=\begin{dcases}
\frac{n}{2}-\frac{R}{4}+ \frac{R}{4}\frac{\binom{N-S}{n}}{\binom{N}{n}}\GaussF{-n}{-S}{N-S-n+1}{-1} & 0\le n\le\frac{N}{2},\\
\frac{n}{2}-\frac{R}{4}+(-1)^S\frac{R}{4}\frac{\binom{N-S}{N-n}}{\binom{N}{N-n}}\GaussF{n-N}{-S}{n-S+1}{-1} & \frac{N}{2}\le n\le N.
\end{dcases}\nonumber
\end{align}
\end{proof}

\begin{proof}[Proof of Theorem~\ref{thm2}]
The variance of $X_n=\sum_{i=1}^R\lfloor Y_{n,i}/2\rfloor$ can be efficiently calculated by using Lemmas~\ref{lemma1}\ref{itm:lemma1-b} and \ref{itm:lemma1-c}.
\begin{align}
\Var{X_n}&=\sum_{i=1}^R \Var{\left\lfloor\frac{Y_{n,i}}{2}\right\rfloor}+\mathop{\sum_{i=1}^R\sum_{j=1}^R}_{i\ne j}\Cov{\left\lfloor\frac{Y_{n,i}}{2}\right\rfloor}{\left\lfloor\frac{Y_{n,j}}{2}\right\rfloor}\nonumber\\
&=\frac{R}{16}+\frac{R(R-1)}{16}G_2(-1)-\frac{R^2}{16}(G(-1))^2.
\label{eq:proof_V}
\end{align}
By substituting $G$ in Eq.~\eqref{eq:pgf} and $G_2$ in Eqs.~\eqref{eq:G2-1} and \eqref{eq:G2-2}, the proof is complete.
\end{proof}

\section{Asymptotic forms for the mean and variance}
We investigate suitable scaling limits of the mean $\E{X_n}$ and variance $\Var{X_n}$ in $R\to\infty$.
More precisely, we take $R\to\infty$ (i.e., $N=SR\to\infty$) and $n\to\infty$ while keeping $x=n/N$ constant ($0<x<1$).

\begin{lemma}[Asymptotic approximation of the probability-generating function $G(z)$]\label{lemma2}
For $R\to\infty$ and $n\to\infty$ with $x=n/N$ fixed, the probability-generating function of $Y_{n,i}$ is written as
\begin{equation}
G(z)=(1-x+xz)^S-\frac{S(S-1)}{2N}x(1-x)(z-1)^2(1-x+xz)^{S-2}+\Omicron(N^{-2}).
\label{eq:G_asymptotic}
\end{equation}
\end{lemma}

\begin{proof}
The probability mass function of $Y_{n,i}$ in Eq.~\eqref{eq:hg_pmf} is written as
\[
\Prob{Y_{n,i}=k}=\binom{S}{k}\frac{(N-S)!}{N!}\frac{n!}{(n-k)!}\frac{(N-n)!}{(N-n-S+k)!}.
\]
We can use Stirling's approximation
\[
M!=\sqrt{2\pi M}\left(\frac{M}{e}\right)^M\left(1+\frac{1}{12M}+\Omicron(M^{-2})\right)
\]
to obtain
\begin{align*}
\frac{(N-S)!}{N!}&=N^{-S}\left(1+\frac{S(S-1)}{2N}+\Omicron(N^{-2})\right),\\
\frac{n!}{(n-k)!}&=(Nx)^k\left(1-\frac{k(k-1)}{2Nx}+\Omicron(N^{-2})\right),\\
\frac{(N-n)!}{(N-n-S+k)!}&=(N(1-x))^{S-k}\left(1-\frac{(S-k)(S-k-1)}{2N(1-x)}+\Omicron(N^{-2})\right).
\end{align*}
Thus,
\[
\Prob{Y_{n,i}=k}=\binom{S}{k}x^k(1-x)^{S-k}\left(1+\frac{S(S-1)}{2N}-\frac{k(k-1)}{2Nx}-\frac{(S-k)(S-k-1)}{2N(1-x)}+\Omicron(N^{-2})\right).
\]
The $\Omicron(N^{-1})$ terms are necessary to calculate the scaling limit of $\Var{X_n}$.
Ord~\cite{Ord} obtained an approximation of the hypergeometric distribution equivalent to this equation.
If the $N\to\infty$ limit is taken, the hypergeometric distribution converges to the binomial distribution~\cite{Johnson}:
\[
\Prob{Y_{n,i}=k}\to\binom{S}{k}x^k(1-x)^{S-k}.
\]

Then, the probability-generating function $G(z)=\sum_{k=0}^S \Prob{Y_{n,i}=k}z^k$ can be calculated by using the probability-generating function of the binomial distribution
\[
\sum_{k=0}^S \binom{S}{k}x^k(1-x)^{S-k} z^k=(1-x+xz)^S
\]
and the relations
\[
\sum_{k=0}^S \binom{S}{k}k(k-1)x^k(1-x)^{S-k}z^k=S(S-1)x^2 z^2(1-x+xz)^{S-2}
\]
and
\[
\sum_{k=0}^S \binom{S}{k}(S-k)(S-k-1)x^k(1-x)^{S-k}z^k=S(S-1)(1-x)^2(1-x+xz)^{S-2}.
\]
\end{proof}

\begin{theorem}[Asymptotic forms of $\E{X_n}$ and $\Var{X_n}$]\label{thm3}
In the limits $R\to\infty$ and $n\to\infty$ with $x=n/N$ fixed, the mean and variance of $X_n$ scaled by $R$ respectively become
\begin{equation}
\frac{\E{X_n}}{R}\to \frac{S}{2}x-\frac{1}{4}+\frac{(1-2x)^S}{4}\eqqcolon \mathcal{E}(x)
\label{eq:asymptotic_mean}
\end{equation}
and
\begin{equation}
\frac{\Var{X_n}}{R}\to\frac{1}{16}+\frac{S-1}{16}(1-2x)^{2S}-\frac{S}{16}(1-2x)^{2S-2}\eqqcolon\mathcal{V}(x).
\label{eq:asymptotic_var}
\end{equation}
\end{theorem}

\begin{proof}
From Eq.~\eqref{eq:proof_E},
\[
\frac{\E{X_n}}{R}=\frac{n}{2R}-\frac{1}{4}+\frac{G(-1)}{4}\to\frac{S}{2}x-\frac{1}{4}+\frac{(1-2x)^S}{4},
\]
where the leading term of $G(-1)$ is
\[
G(-1)=(1-2x)^S+\Omicron(N^{-1})
\]
from Lemma~\ref{lemma2}.
Calculating $\E{X_n}$ requires only the leading term of $G(-1)$.

Next, $\Var{X_n}/R$ can be obtained by using Eq.~\eqref{eq:proof_V}:
\[
\frac{\Var{X_n}}{R}=\frac{1}{16}+\frac{R-1}{16}G_2(-1)-\frac{R}{16}(G(-1))^2.
\]
The asymptotic form of $G_2(z)$ can be obtained by simply replacing $S$ with $2S$ in Eq.~\eqref{eq:G_asymptotic}:
\[
G_2(z)=(1-x+xz)^{2S}-\frac{2S(2S-1)}{2N}x(1-x)(z-1)^2(1-x+xz)^{2S-2}+\Omicron(N^{-2}).
\]
Substituting $G(-1)$ and $G_2(-1)$ completes the proof.
\end{proof}

The function $\mathcal{E}(x)$ defined in Eq.~\eqref{eq:asymptotic_mean} is monotonically increasing in $0\le x\le1$ and satisfies $\mathcal{E}(0)=1$, $\mathcal{E}'(0)=0$, and $\mathcal{E}(1)=\left\lfloor S/2\right\rfloor$ for any $S$ and
\[
\mathcal{E}'(1)=
\begin{dcases}
S & \text{($S$ is even)},\\
0 & \text{($S$ is odd)}.
\end{dcases}
\]

It is expected that $\E{X_n}/R$ for large (but finite) $R$ becomes close to the asymptotic form of Eq.~\eqref{eq:asymptotic_mean}.
Figure~\ref{fig3} shows the rescaled graph of Fig.~\ref{fig1} with $\E{X_n}/R$ and $n/N(=x)$ for (a) $S=4$ and (b) $S=5$.
The solid curve represents the asymptotic function $\mathcal{E}(x)$.
The points overlap almost perfectly along the curve of $\mathcal{E}(x)$ including small $R$ such as $R=3$ and $5$.
Depending on whether $S$ is even or odd, the curvature of $\mathcal{E}(x)$ near $x\approx1$ changes.
This observation is generally justified by
\[
\mathcal{E}''(1)=(-1)^S S(S-1).
\]
In other words, $\mathcal{E}(x)$ near $x=1$ is concave upward for even $S$ and concave downward for odd $S$.

\begin{figure}[t!]\centering
\raisebox{3.9cm}{\small{(a)}}
\includegraphics[scale=0.8]{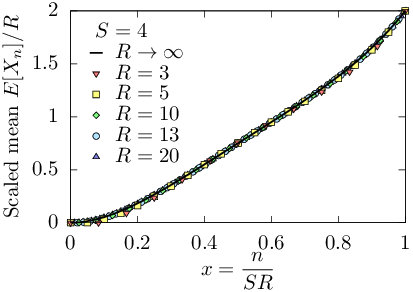}
\hspace{5mm}
\raisebox{3.9cm}{\small{(b)}}
\includegraphics[scale=0.8]{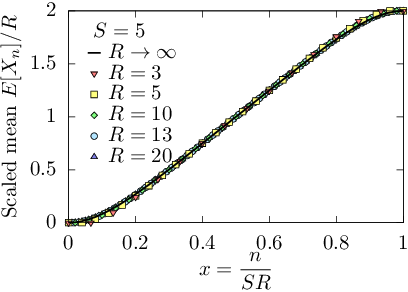}
\caption{
Graph of $\E{X_n}/R$ as a function of $x=n/N$ for (a) $S=4$ and (b) $S=5$.
The solid curve represents the asymptotic form $\mathcal{E}(x)$ given in Eq.~\eqref{eq:asymptotic_mean}.
}
\label{fig3}
\end{figure}

The function $\mathcal{V}(x)$ defined in Eq.~\eqref{eq:asymptotic_var} is symmetric with respect to $x=1/2$ (i.e., $\mathcal{V}(1/2+x)=\mathcal{V}(1/2-x)$).
This function satisfies $\mathcal{V}(0)=\mathcal{V}(1)=0$, $\mathcal{V}(1/2)=1/16$, and $\mathcal{V}'(0)=\mathcal{V}'(1)=\mathcal{V}'(1/2)=0$ for any $S$.
The maximum of $\mathcal{V}(x)$ in $0\le x\le1$ is attained at $x=1/2$.

Figure~\ref{fig4} shows a rescaled graph of Fig.~\ref{fig2} with $\Var{X_n}/R$ and $n/N(=x)$ for (a) $S=4$ and (b) $S=5$.
$\mathcal{V}(x)$ is shown by the solid curve.
Points lie close to the curve of $\mathcal{V}(x)$.
Compared to Fig.~\ref{fig3}, however, points for $R=3$ and $5$ exhibit a relatively large deviation from $\mathcal{V}(x)$.

\begin{figure}[t!]\centering
\raisebox{3.9cm}{\small{(a)}}
\includegraphics[scale=0.8]{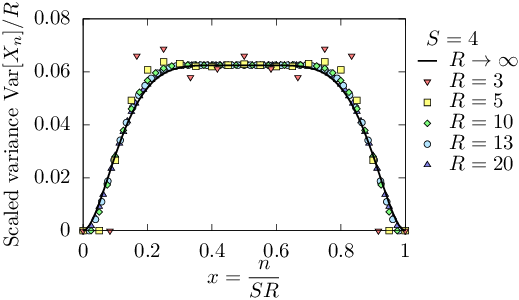}
\hspace{5mm}
\raisebox{3.9cm}{\small{(b)}}
\includegraphics[scale=0.8]{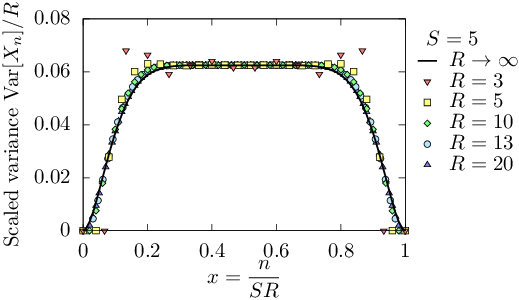}
\caption{
Graph of $\Var{X_n}/R$ as a function of $x=n/N$ for (a) $S=4$ and (b) $S=5$.
The solid curve represents the asymptotic form $\mathcal{V}(x)$ in Eq.~\eqref{eq:asymptotic_var}.
}
\label{fig4}
\end{figure}

\bibliographystyle{apsrev4-2}
\bibliography{bibs}

\end{document}